\documentclass[12pt]{article}
\usepackage{amssymb,amsmath,amsthm,amsfonts}
\usepackage{mathrsfs}
\usepackage{pstricks,pst-plot,epsfig}
\makeatletter
\def\@fnsymbol#1{\ensuremath{\ifcase#1\or 1\or 2\fi}}
\makeatother

\def\sha{\mathcal{A}}

\def\shd{\mathcal{D}}
\def\she{\mathcal{E}}

\def\shh{\mathcal{H}}
\def\shm{\mathcal{M}}
\def\sho{\mathcal{O}}
\def\shn{\mathcal{N}}

\newcommand{\C}{\mathbb{C}}

\newcommand{\R}{\mathbb{R}}

\newcommand{\Z}{\mathbb{Z}}
\newtheorem{theorem}{Theorem}[section]
\newtheorem{proposition}[theorem]{Proposition}
\newtheorem{lemma}[theorem]{Lemma}
\newtheorem{corollary}[theorem]{Corollary}

\theoremstyle{definition}
\newtheorem{definition}[theorem]{Definition}

\newtheorem{example}[theorem]{Example}

\newtheorem{remark}[theorem]{Remark}

\def\Op{\text{Op}(X_{sa})}
\begin{document}

\author{ Ana Rita Martins, Teresa Monteiro Fernandes\footnote{The research of the author was supported by
Funda\c c{\~a}o para a Ci{\^e}ncia e Tecnologia and Programa
Ci{\^e}ncia, Tecnologia e Inova\c c{\~a}o do Quadro
Comunit{\'a}rio de Apoio.}
}
\title{Formal extension of the Whitney functor and duality}
\date{\today}

\maketitle

\begin{abstract}
 On a complex manifold we introduce the formal extension of the Whitney functor and the polynomial extension of the tempered cohomology functor, and prove  a natural topological duality between them.
\end{abstract}
\section{Introduction}

\hspace*{\parindent}In \cite{KS2}, Kashiwara and Schapira introduced the Whitney and the tempered cohomology functors on the subanalytic site $X_{sa}$  associated to a complex manifold $X$, giving  a meaning to $\mathcal{C}^{\infty,\text{w}}_{X_{sa}}$ and to $\sho_{X_{sa}}^{\text{w}}$ (Whitney $\mathcal{C}^{\infty}$ and holomorphic functions),  to $\mathcal{D}b_{X_{sa}}^t$ and to $\sho_{X_{sa}}^t$ (tempered distributions and tempered holomorphic functions) as sheaves (in the derived sense) on $X_{sa}$. We also refer to \cite{L} for a detailed study on sheaves on the subanalytic site.

Let $\shd_X$ denote the sheaf of linear differential operators on $X$.

The duality theorem of Kashiwara and Schapira (Theorem 6.1 of \cite{KS2}) states that taking global sections of the Whitney functor   gives a complex of topological  $\C$-vector spaces of type FN, in duality with the complex of compactly supported sections of  tempered cohomology, this last complex being of topological DFN type. Thus they generalized to the framework of $\shd_X$-modules the classical duality between $\mathcal{C}^{\infty}$-functions and distributions with compact support.

Influenced by several papers on Deformation Quantization, it became a natural question to extend various results in $\shd$-Module theory to the case of the formal extension of $\shd_X$ by a parameter $\hbar$, that is, to $\shd_X[[\hbar]]$-modules. We refer, in particular,  \cite{DGS} and \cite{D}.

In this paper, we are interested in extending the above described  topological duality  to this new framework.
The topological space $\C[[\hbar]]$ is a FN space. Its topological dual is classically known as the space $\C[\hbar]$, but here, as we shall show, it is more natural to consider its dual as being  the quotient of the fraction field $\C((\hbar))$ by $\C[[\hbar]]$,  $\C((\hbar))/ \C[[\hbar]]$, which will be denoted $\C^{[\hbar]}$ for short.  Clearly $\C^{[\hbar]}$ is isomorphic (as a $\C[[\hbar]]$-module with torsion) to the polynomial ring $\C[\hbar ^{-1}]$ with the relations $\hbar\times 1=0$, $\hbar\times \hbar^{-1}=0,\,\, \hbar\times\hbar^{-j}=\hbar^{-j+1}$, for $j>1$. For $f\in\C^{\hbar}$ and $g\in \C^{[\hbar]}$  the duality is then given by $\langle f, g\rangle= Res_{\hbar=0} fg$.

 We shall need to work with the (left) derived functor of the tensor product of sheaves of $\C^{\hbar}$-modules  by the $\C^{\hbar}$-module $\C^{[\hbar]}$.   The extension (of a $\C[[\hbar]]$-module) by $\C^{[\hbar]}$
will be  called ``polynomial extension" for short.


After reviewing notations and necessary results on the subanalytic site in Section 2 and on topological duality in Section 3, Section 4 is dedicated to introduce and study the formal extension of the Whitney functor, as a functor on the category of $\R$-constructible objects over $\C^{\hbar}$.  Inspired  by the formal extension of the functor of tempered cohomology performed by \cite{DGS}, we use the theory of sheaves on the subanalytic site and define the sheaves $\mathcal{C}_{X_{sa}}^{\infty, \text{w}, \hbar}$ (of  Whitney $\mathcal{C}^{\infty, \hbar}$ functions) and $\sho_{X_{sa}}^{\text{w}, \hbar}$ (of Whitney holomorphic functions) on $X_{sa}$.

 Section 5 is dedicated to introduce and study the polynomial extension functor. Namely, we introduce the notion of cohomologically $\hbar$-torsion object, as a kind of dual of the notion of cohomologically $\hbar$-complete introduced in \cite{KS4}. We construct the polynomial extensions of  tempered cohomology on the category of $\R$-constructible objects over $\C^{\hbar}$, following the same technique as in the preceding section, and we define the sheaves $Db_{X_{sa}}^{t, [\hbar]}$ and $\sho_{X_{sa}}^{t, [\hbar]}$.

We obtain  comparison results (cf. Proposition \ref{P:9}, Proposition \ref{P:17}) for formal and polynomial extensions of regular holonomic $\shd^{\hbar}$-modules as an application of the results in \cite{KS2}.

In Section 6 we state and prove the topological duality in the framework of the new functors. More precisely,
in Proposition \ref{P:21} we prove that taking global sections of the formal extension of the Whitney functor still leads to a complex of FN spaces, and that taking compactly supported sections of the polynomial extension of the tempered cohomology still leads to a complex of DFN spaces. Moreover we obtain a topological duality between these complexes.

 Theorem \ref{T:22} establishes the topological FN type of the complex of solutions of coherent $\shd_X^{\hbar}$-modules with values in the formal extension of the Whitney product, as well as the DFN type of the complex of compactly supported solutions of coherent $\shd_X^{\hbar}$-modules with values in the polynomial extension of the tempered holomorphic functions.   By Proposition \ref{P:1} these  complexes  are mutually dual.

It is a pleasure to thank M. Kashiwara and P. Schapira for their enlightening  suggestions. We also thank Luca Prelli for his comments on the subanalytic site and Stephane Guillermou for useful discussions.

\section{Review and complements on sheaves on the subanalytic site and formal extensions}
\hspace*{\parindent}For the background on sheaves on the subanalytic site and Propositions \ref{P:1} and \ref{P:3} below an  we refer to \cite{KS3} ( also to \cite{L}
for a detailed study).
For the background on formal extensions we refer to \cite{KS4} and also \cite{DGS} for the formal extension of the temperate cohomology functor.
\subsection{Sheaves on the subanalytic site}

\hspace*{\parindent}Let $\mathbb{K}$ be a unital Noetherian ring which we assume to have finite global dimension. In practice, throughout this paper, $\mathbb{K}$ will be $\C$ or $\C[[\hbar]]$.

Given a sheaf $\mathcal{R}$ of $\mathbb{K}$-algebras on a topological space $X$, or more generally, on a site, we denote by Mod$(\mathcal{R})$ the category of left $\mathcal{R}$-modules.  We use the notations $D(\mathcal{R})$ for the derived category of Mod($\mathcal{R}$) and $D^b(\mathcal{R})$ for its bounded derived category. We denote by $D^b_{coh}(\mathcal{R})$ the full triangulated subcategory of $D^b(\mathcal{R})$ consisting of objects with coherent cohomology.

For  a real analytic manifold $X$, we denote by Mod$_{\R-c}(\mathbb{K}_X)$ (resp. Mod$^c_{\R-c}(\mathbb{K}_X)$) the category of $\R$-constructible sheaves (resp. with compact support) of $\mathbb{K}$-modules. We denote $D^b_{\R-c}(\mathbb{K}_X)$ the bounded derived category of Mod$_{\R-c}(\mathbb{K}_X)$ of objects having $\R$-constructible cohomology. For $F\in D^b_{\R-c}(\mathbb{K}_X)$ we note $D^{'}(F)$ the object $R\shh om_{\mathbb{K}_X}(F, \mathbb{K}_X)$.

We also denote by $\mathcal{D}b_X$ the sheaf of Schwartz distributions, by $\mathcal{C}^{\infty}_X$ the sheaf of $\mathcal{C}^{\infty}$ functions, by $\sha_X$ the sheaf of real analytic functions and by $\sha_X^{\nu}$ the sheaf of real analytic densities.

 Let $X_{sa}$ denote the associated subanalytic site to a real analytic manifold $X$, that is, the presite $\Op$ of subanalytic open subsets of $X$ endowed with the Grothendieck topology for which the coverings are those admitting a finite sub-covering. Recall that one has a natural morphism of sites $\rho:X\to X_{sa}$ which induces functors
$$\text{Mod}(\C_X)\overset{\rho_*}{\underset{\rho^{-1}}{\rightleftarrows}}\text{Mod}(\C_{X_{sa}}),$$
and we still denote by $\rho_*$ the restriction of $\rho_*$ to
Mod$_{\R\text{-c}}(\C_X)$ and to Mod$_{\R\text{-c}}^c(\C_X)$. Recall that $\rho_*$ is left exact and that it induces an exact functor on Mod$_{\R-c}(\C_X)$. Thereby we identify $F$ and $\rho_*(F)$. Moreover, the functor $\rho^{-1}$ is left adjoint to $\rho_*$ and
 $\rho^{-1}$ admits a left adjoint, denoted by $\rho_!$. Recall that $\rho_!$F is the sheaf on $X_{sa}$ associated to the presheaf $U\to F(\overline{U})$, for $U\in\Op$.

\begin{proposition}\label{P:1}
Let $\{F_i\}_{i\in I}$ be a filtrant inductive system in $\mathrm{Mod}(\C_{X_{sa}})$ and let $U$ be a relatively compact subanalytic open subset of $X$. Then: $$\underset{\underset{i\in I
}{\longrightarrow}}{\lim }\Gamma(U;F_i)\xrightarrow{\sim} \Gamma(U;\underset{\underset{i\in I
}{\longrightarrow}}{\lim }F_i).$$
\end{proposition}

\begin{proposition}\label{P:3}
Let $F=\underset{\underset{i\in I
}{\longrightarrow}}{\lim }F_i$ with $F_i\in\mathrm{Mod}(\C_{X_{sa}})$ and let $G\in D^b_{\R-c}(\C_X)$. One has: $$R^k\mathcal{H}{om}(G,F)\simeq \underset{\underset{i\in I
}{\longrightarrow}}{\lim }R^k\mathcal{H}{om}(G,F_i),$$ for each $k\in\Z$.
\end{proposition}

Recall that the functor $\rho_*$ does not commute with direct sums in general. However, this is true when considering a direct sum of copies of a same $\R$-constructible module, which will suffice for our purposes. For a set of indexes $I$, let us set $F^{\oplus I}:=\oplus_{i\in I} F_i$ with $F_i=F$.

More precisely, we have the following property:

\begin{lemma}\label{L:I}
Let  $F$ be an object of $\mathrm{Mod}_{\R-c}(\C_X)$. Then for any set $I$ of indexes  one has: $$\rho_*(F^{\oplus I})\simeq (\rho_*F)^{\oplus I}$$ in $\mathrm{Mod}_{X_{sa}}(\C_X)$.
\end{lemma}

\begin{proof}
 Every object  $F\in\text{Mod}_{\R-c}(\C_X)$ admits a finite resolution of the form:
 \begin{equation}
0\to \underset{i\in I_1}{\oplus}\C_{U_{1,i}}\to\cdots\to \underset{i\in I_m}{\oplus}\C_{U_{m,i}}\to 0,
\end{equation}
 by locally finite families $\{U_{k,i_k}\}_{k,i_k}$ of relatively compact open subanalytic sets of $X$ (see Appendix of {KS2}). Since the functor $(\cdot)^{\oplus I}$ is exact on Mod$(\C_X)$,  we have a quasi-isomorphism:
 \begin{equation}
F^{\oplus I}\underset{\text{qis}}{\simeq} 0\to \underset{i\in I_1}{\oplus}(\C_{U_{1,i}})^{\oplus I}\to\cdots\to \underset{i\in I_m}{\oplus}(\C_{U_{m,i}})^{\oplus I}\to 0
\end{equation}
in Mod$(\C_X)$.
On the other hand,  $\R$-constructible sheaves are injective with respect to the functor $\rho_*$, which entails a quasi-isomorphism
 \begin{equation}
\rho_*F\underset{qis}{\simeq} 0\to \underset{i\in I_1}{\oplus}\rho_*\C_{U_{1,i}}\to\cdots\to \underset{i\in I_m}{\oplus}\rho_*\C_{U_{m,i}}\to 0,
\end{equation} in Mod$(\C_{X_{sa}})$,
 hence a quasi-isomorphism
 \begin{equation}
 (\rho_*F)^{\oplus I}\underset{qis}{\simeq}0\to \underset{i\in I_1}{\oplus}(\rho_*\C_{U_{1,i}})^{\oplus I}\to\cdots\to \underset{i\in I_m}{\oplus}(\rho_*\C_{U_{m,i}})^{\oplus I}\to 0
\end{equation}
in Mod$(\C_{X_{sa}})$.
Since weakly $\R-$constructible   sheaves are also injective with respect to the functor $\rho_*$, we obtain a quasi-isomorphism

 \begin{equation}
\rho_*(F^{\oplus I})\underset{qis}{\simeq}0\to \underset{i\in I_1}{\oplus}\rho_*(\C_{U_{1,i}})^{\oplus I}\to\cdots\to \underset{i\in I_m}{\oplus}\rho_*(\C_{U_{m,i}})^{\oplus I}\to 0
\end{equation}
in Mod$(\C_{X_{sa}})$.
Therefore, we are reduced to prove that $\rho_*(\C_U^{\oplus I})\simeq (\rho_*\C_U)^{\oplus I}$, for any relatively compact open subanalytic subset $U\subset X$. This will follow if we prove that, for each relatively compact open subanalytic subset $V\subset X$, there exists a finite covering $\{V_i\}_i$ of $V$ by open subanalytic sets $V_i$ such that $\Gamma(V_i;\rho_*(\C_U^{\oplus I}))\simeq \Gamma(V_i;(\rho_*\C_U)^{\oplus I})$.

Indeed, since direct sums are a particular case of inductive limits, by Proposition \ref{P:1} we have, for any relatively compact $\Omega\in\Op$: $$\Gamma(\Omega;(\rho_*\C_U)^{\oplus I})\simeq \Gamma(\Omega;\C_U)^{\oplus I}.$$  Therefore, by the isomorphism $\C_U^{\oplus I}\simeq (\C^{\oplus I})_U$,  we have to prove that there exists a finite covering $\{V_i\}_i$ of $V$ by open subanalytic sets $V_i$ $$
\Gamma(V_i;(\C^{\oplus I})_U)\simeq \Gamma(V_i;\C_U)^{\oplus I}.$$
In the rest of the proof we shall use $K$ to denote either $\C$ or $\C^{\oplus I}$ and we follow the notations of \cite{KS1} for constructibility on a simplicial complex.

Let us consider the subanalytic stratification of $X$ given by $$(U\cap V)\sqcup(U\backslash V)\sqcup(V\backslash U)\sqcup(X\backslash (U\cup V)).$$ By the triangulation theorem (cf \cite{KS1}) there exist a simplicial complex $(S, \Delta)$ and a homeomorphism $i: |S|\to X$ compatible with the stratification above such that $V$ is a finite union of connected open subanalytic sets of the form $i(U(\sigma))=i(\bigcup_{\tau \in\Delta, \tau\supseteq \sigma} |\tau|)$. More precisely,  $V=\bigcup_{i(|\sigma|)\subset V}i(U(\sigma))$. On the other hand, given $\sigma\in\Delta$ such that $i(|\sigma|)\subset V$, and $x\in|\sigma|$, by Proposition 8.1.4 of \cite{KS1} we get: $$\Gamma(i(U(\sigma)); K_U)\simeq \Gamma(U(\sigma); i^{-1}K_U)\simeq \left(i^{-1}K_U\right)_x\simeq (K_U)_{i(x)},$$ since $i^{-1}K_U$ is a weakly $S$-constructible sheaf. Therefore $$\Gamma(i(U(\sigma));K_U)\simeq\begin{cases} K, \text{if $i(|\sigma|)\subset U$}\\  0, \text{if $i(|\sigma|)\nsubseteq U$}\end{cases},$$ which entails the desired isomorphism, taking as $(V_i)$ the covering  by   $(i(U(\sigma)))$ of $V$.
\end{proof}

We shall now give a short overview on the Whitney functor (\cite{KS2}), denoted by $\overset{\text{w}}{\otimes}$, and on the tempered cohomoloy functor, denoted by $t\shh om$ (introduced in \cite{K2} and detailedly studied in \cite{KS2}).

 The Whitney functor, denoted by  $(\cdot)\overset{\text{w}}{\otimes}\mathcal{C}_{X}^{\infty}$, is a functor from $D^b_{\R-c}(\C_X)$ to $D^b(\shd_X)$, inducing an exact functor from $\text{Mod}_{\R-c}(\C_X)$   to $\text{Mod}(\shd_X)$, and such that, for $U$ open subanalytic in $X$, $$\C_U\overset{\text{w}}{\otimes}\mathcal{C}_{X}^{\infty}=\mathcal{I}^\infty_{X,X\backslash U},$$
the sheaf of $\mathcal{C}^\infty$ functions on $X$ vanishing up to infinite order on $X\backslash U$.

If $\mathcal{L}$ is a $\sha_X$-locally free module of finite rank, one sets
$$F\overset{\text{w}}{\otimes}(\mathcal{C}_{X}^{\infty}\underset{\sha_X}{\otimes}\mathcal{L}):=(F\overset{\text{w}}{\otimes}\mathcal{C}_{X}^{\infty})\underset{\sha_X}{\otimes}\mathcal{L}.$$

 The  tempered distribution cohomology functor, denoted by $t\mathcal{H}om(\cdot, \shd b_X)$ a functor from $D^b_{\R-c}(\C_X)$ to $ D^b(\shd_X)$, inducing an exact functor from  $\text{Mod}_{\R-c}(\C_X)$ to $\text{Mod}(\shd_X)$ and such that, for $Z$ closed subanalytic in $X$, $$t\mathcal{H}om(\C_Z, \shd b_X)=\Gamma_Z(\shd b_X),$$
the sheaf of Schwartz distributions supported by $Z$.

One notes $t\mathcal{H}om(\cdot, \shd b_X^{\nu}):=t\mathcal{H}om(\cdot, \shd b_X)\underset{\sha_X}{\otimes} \sha_X^{\nu}.$

These two functors are extended as sheaves on the subanalytic site $X_{sa}$ (\cite{KS4}) as follows:

Let  $\mathcal{C}_{X_{sa}}^{\infty,\text{w}}$ denote the sheaf on $X_{sa}$ of Whitney $\mathcal{C}^{\infty}$-functions, that is, the sheaf defined by: $$U\mapsto \Gamma(X;R\mathcal{H}{om}(\C_U,\C_M)\overset{\text{w}}{\otimes}\mathcal{C}_{X}^{\infty}).$$

Let  $\shd b^t_{X_{sa}}$ denote the sheaf on $X_{sa}$ of tempered distributions, that is, the sheaf defined by: $$U\mapsto \Gamma(X; t\mathcal{H}{om}(\C_U, \shd b_X)).$$

 We have the following isomorphisms in $D^b(\shd_X)$:

For $F\in D^b_{\R-c}(\C_X)$,
$$F\overset{\text{w}}{\otimes}\mathcal{C}_{X}^{\infty} \simeq \rho^{-1}(R\mathcal{H}om(D'(F), \mathcal{C}_{X_{sa}}^{\infty, \text{w}}))$$ and

$$t\mathcal{H}om(F, \shd b_X)\simeq \rho^{-1}(R\shh om(F, \shd b_{X_{sa}}^t)).$$

For a complex analytic manifold $X$, we denote by $\shd_X$ the sheaf of differential operators of finite order, by $\sho_X$ the sheaf of holomorphic functions and by $\Omega_X$ the sheaf of holomorphic differential forms of maximal degree. Considering the complex conjugate structure  in $X$  and denoting it by $\overline{X}$, one defines the following sheaves on $X_{sa}$:

The sheaf of holomorphic Whitney functions, $\sho_{X_{sa}}^{\text{w}}$, given by $$\sho_{X_{sa}}^{\text{w}}=R\shh om_{\rho_!\shd_{\overline{X}}}(\rho_!\sho_{\overline{X}}, \mathcal{C}^{\infty, \text{w}}_{\overline{X}_{sa}}).$$

The sheaf of tempered holomorphic functions, $\sho_{X_{sa}}^t$, given by

$$\sho^t_{X_{sa}}=R\shh om_{\rho_!\shd_{\overline{X}}}(\rho_!\sho_{\overline{X}}, \mathcal{D} b^t_{\overline{X}_{sa}}).$$

\subsection{Review on formal extensions and the formal tempered cohomology functor}

\hspace*{\parindent}Let $\mathcal{R}$ be a $\mathbb{Z}[\hbar]$-algebra such that $\hbar: \mathcal{R}\to \mathcal{R}$ is injective (i.e, $\mathcal{R}$ is free of $\hbar$-torsion.)
We note $\mathcal{R}^{\text{loc}}:=\mathbb{Z}[\hbar, \hbar^{-1}]\otimes_{\mathbb{Z}[\hbar]}\mathcal{R}$, and $\mathcal{R}_0:=\mathcal{R}/\hbar \mathcal{R}$. We obtain the functors $$(\cdot)^{\text{loc}}:\text{Mod}(\mathcal{R})\to \text{Mod}(\mathcal{R}^{\text{loc}}), \shm\to\shm^{\text{loc}}:=\mathcal{R}^{\text{loc}}\otimes_{\mathcal{R}}\shm,$$ which is exact,
and
$$gr_{\hbar}:D(\mathcal{R})\to D(\mathcal{R}_0), \shm\to gr_{\hbar}(\shm)=\mathcal{R}_0\overset{L}{\otimes}_{\mathcal{R}}\shm.$$
Recall that $\shm\in D(\mathcal{R})$ is cohomologically $\hbar$-complete if $R\mathcal{H}om_{\mathcal{R}}(\mathcal{R}^{\text{loc}}, \shm)=0$.
We say that  a $\mathbb{Z}[\hbar]$-module $\shm$ is $\hbar$-complete if $\shm\to\underset{j\geq 0}{\underset{\longleftarrow}{\lim}} \shm/\hbar^j\shm$ is an isomorphism.

\begin{proposition}\label{P121}
The functor $gr_{\hbar}$ is conservative on the category of cohomologically $\hbar$-complete objects, that is, if $\shm\in D(\mathcal{R})$ is cohomologically $\hbar$-complete and $gr_{\hbar}(\shm)=0$, then $\shm=0$.
\end{proposition}
\begin{proposition}\label{P122}
For a given cohomologically $\hbar$-complete object $\shm\in D(\mathcal{R})$, for any $\shn\in D(\mathcal{R})$, $R\shh om_{\mathcal{R}}(\shn, \shm)$ is cohomologically $\hbar$-complete.
\end{proposition}

We set $\mathbb{C}^{\hbar}$ to shortly denote the ring  $\mathbb{C}[[\hbar]]$ of formal power series in the $\hbar$ variable and set $\C^{\hbar, loc}:=\C((\hbar))\simeq \C[\hbar^{-1}, \hbar]]$ the field of fractions of $\C^{\hbar}$.

Recall  the (left exact) functor of formal extension  $(\cdot)^\hbar:\text{Mod}(\mathbb{C}_X)\to \text{Mod}(\mathbb{C}_X^{\hbar})$, defined by $$F\to F^{\hbar}:=\underset{j\geq 0}{\underset{\longleftarrow}{\lim}} (F\otimes \mathbb{C}_X^{\hbar}/\hbar^j\mathbb{C}_X^{\hbar}).$$   We denote by $(\cdot)^{R\hbar}$ its right derived functor.
\begin{proposition}\label{P:123}
For any $F\in D^b(\C_X)$, its formal extension $F^{R\hbar}$ is cohomologically $\hbar$-complete.
\end{proposition}
\begin{proposition}\label{P:4}
Let $\mathcal{I}$ be either a basis of open subsets of a site $X$ or, assuming that $X$ is a locally compact topological space, a basis of compact subsets. Denote by $\mathcal{J}_{\mathcal{I}}$ the full subcategory of $\mathrm{Mod}(\C_X)$ consisting of $\mathcal{I}$-acyclic objects, i.e., sheaves $\mathcal{N}$ for which $H^k(S;\mathcal{N})=0$ for all $k>0$ and all $S\in\mathcal{I}$. Then $\mathcal{J}_{\mathcal{I}}$ is injective with respect to the functor $(\cdot)^\hbar$. In particular, for $\mathcal{N}\in\mathcal{J}_{\mathcal{I}}$, we have $\mathcal{N}^\hbar\simeq \mathcal{N}^{R\hbar}$.
\end{proposition}

 The following result which is contained in \cite{DGS}, Lemma 2.3:

\begin{lemma}\label{L:5}
Assume that $\mathcal{R}$ is  a $\C_X$-algebra.
Then,  for $\mathcal{M}, \mathcal{N}\in D^b(\mathcal{R})$, we have an isomorphism in $D^b(\C_X^{\hbar})$ $$R\mathcal{H}{om}_{\mathcal{R}}(\mathcal{M},\mathcal{N})^{R\hbar}\simeq R\mathcal{H}{om}_{\mathcal{R}}(\mathcal{M},\mathcal{N}^{R\hbar}).$$
\end{lemma}
 We recall the properties of the functor $(\cdot)^{R\hbar}$  proved in \cite{DGS} in view of the subanalytic site.
\begin{lemma}\label{L:6}
\begin{itemize}
\item[(i)] The functors $\rho^{-1}$ and $(\cdot)^{R\hbar}$ commute, that is, for $G\in D^b(\C_{X_{sa}})$ we have $(\rho^{-1}G)^{R\hbar}\simeq \rho^{-1}(G^{R\hbar})$ in $D^b(\C_X^\hbar)$.

\item[(ii)] The functors $R\rho_*$ and $(\cdot)^{R\hbar}$ commute, that is, for $F\in D^b(\C_X)$ we have $(R\rho_*F)^{R\hbar}\simeq R\rho_*(F^{R\hbar})$ in $D^b(\C_{X_{sa}}^\hbar)$.
\end{itemize}
\end{lemma}

\begin{lemma}\label{L:19}
Given $F\in D^b_{\R-c}(\C_X^{\hbar})$, $F$ is isomorphic to a complex:
\begin{equation}\label{E:21}
0\to \underset{i\in I_1}{\oplus}\C_{U_{1,i}}^{\hbar}\to\cdots\to \underset{i\in I_m}{\oplus}\C_{U_{m,i}}^{\hbar}\to 0,
\end{equation}
for locally finite families $\{U_{j,i}\}_{j,i}$ of relatively compact subanalytic open subsets of $X$.

\end{lemma}

\begin{lemma}\label{L:h}
For $F\in D^b_{\R-c}(\C_X)$, we have $F^{R\hbar}\simeq F^{\hbar}\simeq \C_X^{\hbar}\otimes F$.
\end{lemma}
Recall that a resolution of $F$ as in Lemma \ref{L:19} is called an ``almost free" resolution.

The sheaves $\mathcal{C}_{X_{sa}}^{\infty,t,\hbar}, \mathcal{D}b_{X_{sa}}^{t,\hbar}$ and $\sho_{X_{sa}}^{t,\hbar}$ on $X_{sa}$ were studied in \cite{DGS}, and proved to be cohomologically $\hbar$-complete.
The authors also introduced the formal extension of $t\shh om(\cdot, \sho_X)$, the functor of  tempered holomorphic  cohomology, and noted it $TH_{\hbar}(\cdot): D^{b}_{\R-c}(\C^{\hbar}_X)\to D^b(\shd_X)$ by setting
$$F\to TH_{\hbar}(F):=\rho^{-1}R\shh om_{\C^{\hbar}_{X_{sa}}}(\rho_*F, \sho_{X_{sa}}^{t,\hbar}),$$
hence, for any $F\in D^b_{\R-c}(\C^{\hbar}_X)$,  $TH_{\hbar}(F)$ is  cohomologically $\hbar$-complete.
\section{Topological Duality}
 \hspace*{\parindent}We say that a (real or complex) topological vector space is of type FN if it is Fr\'{e}chet nuclear and we use the notation DFN for the strong dual of a Fr\'{e}chet nuclear space. Moreover, we shall say that two complexes $V^\bullet$ and $W^\bullet$ of topological vector spaces of type FN and DFN, respectively, are dual to each other if each entry $W^{-i}$ of $W^\bullet$ is the topological dual of the entry $V^i$ of $V^\bullet$ and the morphism $w^i: W^{-i-1}\to W^i$ is the transpose of $v^i: V^i\to V^{i+1}$.

 Noticing that $\C[[\hbar]]$ is a FN topological $\C$-vector space, it is well known that  $\C[[\hbar]]$ and the quotient $\C((\hbar))/\C[[\hbar]]$ are in perfect duality, the duality being given by $\langle f, g\rangle=Res_{\hbar=0}(f g).$
Namely, given $V$ a FN topological $\C$-vector space, $V^{\hbar}$, being a product, is still FN. In addition, noting $V^*$ its strong topological dual, since the topological dual of a countable direct sum is  the product of the topological duals,  $V^{\hbar}$ and $V^*\otimes \C((\hbar))/\C[[\hbar]]$ are topologically dual to each other, the last one being a DFN space.

\section{Formal extension of the Whitney functor}

\hspace*{\parindent}Let $X$ be a real analytic manifold.
 \begin{definition}\label{D:7}
The sheaf of formal Whitney $\mathcal{C}_X^{\infty}$ functions is the object of $\text{Mod}(\C_{X_{sa}}^{\hbar})$ given by
$$\mathcal{C}_{X_{sa}}^{\infty,\text{w},\hbar}:=(\mathcal{C}_{X_{sa}}^{\infty,\text{w}})^\hbar.$$
\end{definition}

Recall that $R\Gamma(U; \mathcal{C}_{X_{sa}}^{\infty,\text{w}})\simeq R\Gamma(X;D'\C_U\overset{\text{w}}{\otimes}\mathcal{C}_{X}^{\infty})$, for all open subanalytic subsets $U$ of $X$ (see \cite{L}). Therefore, for those $U$ such that $D'\C_U$ is concentrated in degree zero (for instance the so called locally cohomologically trivial (l.c.t) subanalytic open subsets),  $ D'\C_U\overset{\text{w}}{\otimes}\mathcal{C}_{X}^{\infty}$ being a soft sheaf,  $R\Gamma(U; \mathcal{C}_{X_{sa}}^{\infty,\text{w}}) $ is concentrated in degree zero.  Since l.c.t. open subanalytic sets form a basis for the site $X_{sa}$,   we have, by Proposition \ref{P:4}:

\begin{equation}\label{E:1}
\mathcal{C}_{X_{sa}}^{\infty,\text{w},\hbar}\simeq (\mathcal{C}_{X_{sa}}^{\infty,\text{w}})^{R\hbar}.
\end{equation}

Therefore, since formal extensions are cohomologically $\hbar$-complete (see Proposition 2.2 of \cite {DGS}), $\mathcal{C}_{X_{sa}}^{\infty,\text{w},\hbar}$ is cohomologically $\hbar$-complete.

Let  $D_{\hbar}'$ denote the functor $\text{D}^b(\C_X^\hbar)^{\text{op}}\to\text{D}^b(\C_X^\hbar), \ F\mapsto \text{R}\mathcal{H}{om}_{\C_X^\hbar}(F, \C_X^\hbar)$.
\begin{definition}\label{D1}
We  define the ``formal extension of Whitney functor"
$$(\cdot)\overset{\text{w},\hbar}{\otimes}\mathcal{C}_X^\infty: D^b_{\R-c}(\C_X^{\hbar})\to D^b(\C^{\hbar}_X),$$
 as the composition of derived functors  $$ F\mapsto F\overset{\text{w},\hbar}{\otimes}\mathcal{C}_X^\infty:= \rho^{-1}R\mathcal{H}{om}_{\C_{X_{sa}}^{\hbar}}(\rho_* (D_{\hbar}'F),\mathcal{C}_{X_{sa}}^{\infty,\text{w},\hbar}).$$
\end{definition}
Since  $\mathcal{C} ^{\infty, \text{w}} _{X_{sa}}$ belongs to $\text{Mod}(\rho_!\shd_X)$, the functor $\rho^{-1}$ commutes with $(\cdot)^{R\hbar}$ and $\rho^{-1}\circ \rho_!\simeq \text{id}$, $ F\overset{\text{w},\hbar}{\otimes}\mathcal{C}_X^\infty$ is an object of $D^b(\shd_X^{\hbar})$, in other words, $(\cdot)\overset{\text{w},\hbar}{\otimes}\mathcal{C}_X ^\infty$ is a functor from $\text{D}^b_{\R-c}(\C_X^{\hbar})$ to $D^b(\shd_X ^{\hbar}).$ Moreover, $F\overset{\text{w},\hbar}{\otimes}\mathcal{C}_X^\infty$ is cohomologically $\hbar$-complete.
\begin{lemma}\label{L:18}
Let $F\in \mathrm{Mod}_{\R-c}(\C_X)$. Then

$$F^{\hbar}\overset{\mathrm{w},\hbar}{\otimes}\mathcal{C}_X^{\infty}\simeq (F\overset{\mathrm{w}}{\otimes}\mathcal{C}_X^{\infty})^{\hbar}$$ in $\mathrm{Mod}(\shd_X^{\hbar})$, hence it is concentrated in degree zero and is a soft sheaf.

\end{lemma}
\begin{proof}
The result follows  from (\ref{E:1}),  Lemmas \ref{L:5} and \ref{L:6} and Proposition \ref{P:4}.
 \end{proof}

\begin{remark}\label{L:141}

Given $F\in \text{D}^b_{\R-c}(\C_X^{\hbar})$, choosing a resolution $F^{\bullet}$ of $F$ as in (\ref{E:21}), that is, such that   each entry $F^i$ is isomorphic to a locally finite sum of $\mathcal{C}_X^\hbar$-modules of the form $ \C_U^{\hbar}$, with $U\in\Op$, we obtain that $F\overset{\text{w},\hbar}{\otimes}\mathcal{C}_X^\infty$ is isomorphic to a bounded complex $F^\bullet\overset{\text{w},\hbar}{\otimes}\mathcal{C}_X^\infty$ which provides a soft resolution of  $F\overset{\text{w},\hbar}{\otimes}\mathcal{C}_X^\infty$.
Therefore, for any  open subanalytic set $U$ in $X$
 $$R\Gamma(U; F\overset{\text{w},\hbar}{\otimes}\mathcal{C}_X^\infty)\simeq \Gamma(U;  F^{\bullet}\overset{\text{w},\hbar}{\otimes}\mathcal{C}_X^\infty).$$
\end{remark}

When $X$ is a complex analytic manifold, denoting by $\overline{X}$ the complex conjugate manifold of $X$ and by $X^{\R}$ the underlying real analytic manifold identified with the diagonal of $X\times \overline{X}$, we may also define the $\hbar$-version of the sheaf   $\sho_X^{\text{w}}$  on $X_{sa}$, by setting:
\begin{definition}\label{D:8}

$$\sho_{X_{sa}}^{\text{w},\hbar}:=R\mathcal{H}{om}_{\rho_!\shd_{\overline{X}}}(\rho_!\sho_{\overline{X}},\mathcal{C}_{X^\R_{sa}}^{\infty,\text{w},\hbar}).$$
\end{definition}

As a consequence of (\ref{E:1}) together with Lemma \ref{L:5} we have:
\begin{equation}\label{E:10}
\sho_{X_{sa}}^{\text{w},\hbar}\simeq(\sho_{X_{sa}}^{\text{w}})^{R\hbar}.
\end{equation}

We now introduce the formal extension of the holomorphic Whitney functor, $(\cdot)\overset{\text{w},\hbar}{\otimes}\mathcal{O}_X: \text{D}^b_{\R-c}(\C_X^\hbar)\to \text{D}^b(\shd_X^\hbar)$, by setting: $$F\overset{\text{w},\hbar}{\otimes}\mathcal{O}_X:= \rho^{-1}R\mathcal{H}{om}_{\C_{X_{sa}}^{\hbar}}(\rho_* (D_{\hbar}'F),\mathcal{O}_{X_{sa}}^{\text{w},\hbar})\simeq R\mathcal{H}{om}_{\rho_!\shd_{\overline{X}}}(\rho_!\sho_{\overline{X}}, F\overset{\text{w},\hbar}{\otimes}\mathcal{C}_{X^\R_{sa}}^{\infty}).$$
Therefore, for any $F\in D^b_{\R-c}(\C^{\hbar}_X), F\overset{\text{w},\hbar}{\otimes}\mathcal{O}_X$ is cohomologically $\hbar$-complete. Moreover, by Lemmas  \ref{L:5} and \ref{L:6} we get
\begin{corollary}\label{R:8}
Let $F\in D^b_{\R-c}(\C_X)$. Then
$$F^{\hbar}\overset{\mathrm{w},\hbar}{\otimes}\sho_X\simeq (F\overset{\mathrm{w}}{\otimes}\sho_X)^{R\hbar}$$ in $D^b(\shd_X^{\hbar}).$

\end{corollary}

Recall that one notes by $D^b_{rh}(\shd_X)$  the full triangulated category of $D^b_{coh}(\shd_X)$ of the objects having regular holonomic cohomology and by $D^b_{rh}(\shd^{\hbar}_X)$ the full triangulated subcategory of $D^b_{coh}(\shd^{\hbar}_X)$ of the objects $\shm$ such that $gr_{\hbar}(\shm)\in D^b_{rh}(\shd_X)$.
We obtain a comparison result:
\begin{proposition}\label{P:9}
Let $\shm \in D^b_{rh}(\shd^{\hbar}_X)$ and let $F\in D^b_{\R-c}(\C_X)$. Then, the natural morphism:
\begin{equation}\label{E:14}
R\shh{om}_{\shd_X^{\hbar}}(\shm, (F\otimes \sho_X)^{R\hbar})\to R\shh{om}_{\shd_X^{\hbar}}(\shm, F^{\hbar}\overset{\mathrm{w},\hbar}{\otimes} \sho_X)
\end{equation}
is an isomorphism in $D^b(\C^{\hbar}_X)$.
\end{proposition}
\begin{proof}
Since both sides of the morphism (\ref{E:14}) are cohomologically $\hbar$-complete  by Propositions \ref{P122} and \ref{P:123}, by Proposition \ref{P121} it is enough to apply $gr_{\hbar}$ and then the result follows  by Corollary \ref{R:8} and Corollary 6.2 of \cite{KS2} which proves the isomorphism $$R\shh{om}_{\shd_X}(gr_{\hbar}(\shm), F\otimes \sho_X)\simeq R\shh{om}_{\shd_X}(gr_{\hbar}(\shm), F\overset{\text{w}}{\otimes} \sho_X).$$
\end{proof}

Remark that Theorem A.9 of \cite{KS2} entails the existence of almost free resolutions in the framework of coherent $\shd_X^{\hbar}$-modules:
\begin{proposition}\label{P:201}
Let $\shm\in D^b_{coh}(\shd_X^{\hbar})$. Then there exist a family $(U_j)_{j\in J}\in\Op $, a complex $\mathcal{L}^{\bullet}\in D^b_{coh}(\shd_X^{\hbar})$ and  a quasi-isomorphism $\mathcal{L}^{\bullet}\to\shm$, such that:

(i) Each $U_j$ is relatively compact,

(ii) Each entry $\mathcal{L}^i$ of $\mathcal{L}$ is a locally finite direct sum of the form $\underset{j\in J_i}{\oplus} {\shd_X ^{\hbar}}_{U_j}$, where $J=\bigcup_{i}J_i$.
\end{proposition}

\section{Polynomial extension of tempered cohomology}

\subsection{The functor of polynomial extension}
\hspace*{\parindent}We set $\C^{[\hbar]}:=\C((\hbar))/\C^{\hbar}.$

For $j\in\Z_{>0}$, let us note by $\C^{[\hbar]}_j$ the image of $\C^{\hbar}+\C[\hbar^{-1}]_j$ in the quotient $\C((\hbar))/\C_X^{\hbar}$, where $\C[\hbar^{-1}]_j$ denotes the set of complex polynomials of degree at most $j$ in the $\hbar^{-1}$ variable, in other words, $\C^{[\hbar]}_j=(\hbar^{-j}\C^{\hbar})/\C^{\hbar}$. Hence, as a $\C$-vector space, $\C^{[\hbar]}$ is isomorphic to the polynomial ring $\C[\hbar^{-1}]$. However, as a quotient of $\C((\hbar))$,  $\C^{[\hbar]}$ is not a ring and we shall keep in mind its $\C^{\hbar}$-module structure.

Given a sheaf $F$ of $\C^{\hbar}$-modules on a topological space $X$, or more generally, a sheaf in $\text{Mod}(\C_{X}^{\hbar})$, for a site $X$ (we shall not distinguish these situations, unless otherwise explicited), one sets  $F^{loc}:=F\underset{\C_X^{\hbar}}{\otimes}\C_X((\hbar))$. Clearly $(\cdot)^{loc}$ is an exact functor.

 For $F\in D^{b}(\C_X^{\hbar})$ we set $F^{loc}:=\C_X((\hbar))\otimes_{\C_X^{\hbar}} F$.

\begin{definition}\label{D:10}
Given a sheaf $F$ of $\C^{\hbar}$-modules, we denote by $F^{[\hbar]}$ the sheaf $F\underset{\C_X^{\hbar}}{\otimes}\C_X^{[\hbar]} $.
\end{definition}
Hence $(\cdot)^{[\hbar]}$ defines a right exact functor on $\text{Mod}(\C_{X}^{\hbar})$ and we note $(\cdot)^{L[\hbar]}$ its left derived functor.
By construction $F^{[\hbar]}$ is a $\C_X^{\hbar}$-module with torsion and, when $F$ has no $\hbar$-torsion, $F^{[h]}$ is isomorphic to $F^{loc}/F$. Note that $F^{[\hbar]}$ is, as a sheaf of $\C^{\hbar}$-modules, isomorphic to $\underset{j}{\underset{\rightarrow}{\lim}} F\underset{\C^{\hbar}_X}{\otimes}{\C^{[\hbar]}_X}_j$. We set $F^{[\hbar]}_j:=F\underset{{\C^{\hbar}_X} }{\otimes} {\C^{[\hbar]}_X}_j$.
Hence  $$\Gamma_c(X;F^{[\hbar]})\simeq\underset{j}{\underset{\rightarrow}{\lim}}\Gamma_c(X;F^{[\hbar]}_j).$$ We also have  $(F^{[\hbar]})_U\simeq F_U^{[\hbar]}$, for any open subset $U$ of $X$.

Clearly, for $F\in D^b(\C_X^{\hbar})$ one has a quasi isomorphism $$F^{L[\hbar]}=F\overset{L}{\otimes}_{\C_X^{\hbar}} \C_X^{[\hbar]} \underset{QIS}{\to} \{0\to F\to F^{loc}\to 0\}.$$

 \begin{definition}\label{D:201}We say that $F\in D^b(\C_X^{\hbar})$ is \textit{cohomologically} $\hbar$-\textit{torsion} if $$F^{loc}=0.$$
 \end{definition}
Therefore $F^{L[\hbar]}$ is obviously cohomologically $\hbar$-torsion.
Indeed $$F\overset{L}{\otimes}_{\C_X^{\hbar}} \C_X^{[\hbar]}\otimes_{\C_X^{\hbar}}\C_X((\hbar))=0$$ because $\C_X^{[\hbar]}\otimes_{\C_X^{\hbar}}\C_X((\hbar))=0$.

Let us note $gr_{\hbar}F:=\C_X\overset{L}{\otimes}_{\C_X^{\hbar}} F$.




   \begin{lemma}\label{L:200}  Assume that $gr_{\hbar}=0$ and that $F$ is cohomologically $\hbar$-torsion. Then $F=0$, in other words,
   the functor $gr_{\hbar}$ is conservative on the subcategory of $D^b(\C_X^{\hbar})$ of cohomologically $\hbar$-torsion objects.
   \end{lemma}
\begin{proof}
1) Let us start by assuming that $F$ is concentrated in degree $0$. Then, since $\hbar$ is invertible on $F$, $F$ is a $\C_X[\hbar^{-1}]$-module, a fortiori a $\C((\hbar))$-module. It follows that $F\simeq F^{loc}$, hence $F=0$.

2) To treat the general case, since $\shh^j(F^{loc})\simeq \shh^j(F)^{loc}$, the result follows by 1) applied to  $\shh^j(F)$.\end{proof}
The first two following results are clear:

 \begin{lemma}\label{L:203}
 Assume that $F$ is cohomologically $\hbar$-torsion, Then, for any $G\in D^b(\C_X^{\hbar})$, $G\overset{L}{\otimes}_{\C_X^{\hbar}} F$ is cohomologically $\hbar$-torsion.
 \end{lemma}

 \begin{lemma}\label{L:204}
 Given $\shn\in D^b_{coh}(({\shd_X^{\hbar}})^{op})$, and $F\in D^b(\shd_X^{\hbar})$, with $F$ cohomologically $\hbar$-torsion in $D^b(\C^{\hbar})$, then $\shn\overset{L}{\otimes}_{\shd_X^{\hbar}}F$ is cohomologically $\hbar$-torsion.
\end{lemma}

\begin{remark}

For any  $F\in \text{Mod}_{\R-c}(\C_X)$, $F^{\hbar}$ is $\hbar$-torsion free, hence $((\cdot)^{\hbar})^{[\hbar]}$ induces an exact functor on $\text{Mod}_{\R-c}(\C_X)$.

\end{remark}

\noindent \textbf{Notation.} For the sake of simplicity, for  $(\cdot)^{\hbar}$-acyclic objects $G\in \text{Mod}(\C_X)$,  $((G)^{\hbar})^{[\hbar]}$ will  be denoted by $G^{[\hbar]}$.

\begin{lemma}\label{L:206}
 Assume that  $H\in D^b(\C_X)$ is $(\cdot)^{\hbar}$-acyclic and that $H^{\hbar}$ is a complex of $\hbar$-torsion free modules. Let $F\simeq H^{[\hbar]}$.
 Then, for any $G\in D^b_{\R-c}(\C_X^{\hbar})$, $R\shh\text{om}_{\C_X^{\hbar}} (G, F)$ is cohomologically $\hbar$-torsion.
 \end{lemma}
\begin{proof}
 Since the statement is of local nature, we may replace $G$ by an almost free resolution $G^{\bullet}$ such that each entry $G^i$ is a locally finite direct sum of sheaves of the form $\C_{\Omega_{ij}}^{\hbar}$, for given open subanalytic sets $\Omega_{ij}$. Therefore
$R\shh\text{om}_{\C_X^{\hbar}} (G, F)$ is isomorphic to $\shh\text{om}(G^{\bullet}, H)^{[\hbar]}$ 
which is cohomologically $\hbar$-torsion, and the result follows.

\end{proof}

\subsection{Polynomial extension of tempered cohomology}
Let us now assume that $X$ is a real analytic manifold and let $X_{sa}$ the associated subanaytic site.
We shall also keep the notation $(\cdot)^{L[\hbar]}$ for the corresponding functor $\text{D}^b(\C_{X_{sa}}^{\hbar})\to\text{D}^b(\C_{X_{sa}}^{\hbar}).$

\begin{lemma}\label{L:12}
\begin{itemize}
\item[(i)] For $F\in  D^b_{\R-c}(\C^{\hbar}_X)$, one has $\rho_*(F^{L[\hbar]})\simeq (\rho_*F)^{L[\hbar]}$.

\item[(ii)] The functors $\rho^{-1}$ and $(\cdot)^{[\hbar]}$ commute, more precisely, for all $F\in  D^b(\C^{\hbar}_X)$, one has $\rho^{-1}(F^ {L[\hbar]})\simeq (\rho^{-1}F)^{L[\hbar]}$ and, for any $G\in D^b(\C_X)$, one has  $\rho^{-1}(F\otimes \C_X^{[\hbar]})\simeq (\rho^{-1}F)\otimes \C_X^{[\hbar]}$.

\end{itemize}
\end{lemma}

\begin{proof}
(i) By Lemma \ref{L:19} it is enough to prove that, for any open subanalytic relatively compact set $U$,  $\rho_*(\C_U^{[\hbar]})\simeq (\rho_*\C_U^{\hbar})^{[\hbar]}$. This is an immediate consequence of Lemma \ref{L:I} and the fact that $\rho_*$ commutes with formal extension.

(ii) The result follows from Lemma \ref{L:19} and the fact that the functor $\rho^{-1}$ commutes with inductive limits.
\end{proof}

We define $\mathcal{D}{b}_{X_{sa}}^{t,[\hbar]}:=(\mathcal{D}{b}_{X_{sa}}^{t, \hbar})^{[\hbar]}$.

Since $\mathcal{D}{b}_{X_{sa}}^{t}$ is $(\cdot)^{\hbar}$-acyclic and  $\mathcal{D}{b}_{X_{sa}}^{t,\hbar}$ is $\hbar$-torsion free we get:

$$\mathcal{D}{b}_{X_{sa}}^{t,[\hbar]}\simeq {(\mathcal{D}{b}_{X_{sa}}^{t})}^{[\hbar]}.$$

\begin{lemma}\label{L:13}
For all $F\in D^b_{\R-c}(\C_X)$
the natural morphism
\begin{equation}\label{E:4}
\rho^{-1}R\mathcal{H}{om}_{\C_{X_{sa}}}(F,\mathcal{D}{b}_{X_{sa}}^{t,[\hbar]})\to  \left(\rho^{-1}R\mathcal{H}{om}_{\C_{X_{sa}}}(F,\mathcal{D}{b}_{X_{sa}}^t)\right)\otimes \C_X^{[\hbar]}
\end{equation} is an isomorphism in $D^b(\C_X^{\hbar})$
and, if $F$ is in degree zero,  $\rho^{-1}R\mathcal{H}{om}_{\C_{X_{sa}}}(F,\mathcal{D}{b}_{X_{sa}}^{t,[\hbar]})$ is concentrated in degree zero.

\end{lemma}

\begin{proof}
For each $k\in\Z$,  by Proposition \ref{P:3},  one has
\begin{equation}\label{E:11}
R^k\mathcal{H}{om}_{\C_{X_{sa}}}(F,\mathcal{D}{b}_{X_{sa}}^{t,[\hbar]})\simeq  R^k\mathcal{H}{om}_{\C_{X_{sa}}}(F,\mathcal{D}{b}_{X_{sa}}^t)\otimes \C_X^{[\hbar]}.
\end{equation}
Therefore, morphism (\ref{E:4}) is an isomorphism.

Now, suppose that $F$ is in degree zero. Then  $R\mathcal{H}{om}_{\C_{X_{sa}}}(F,\mathcal{D}{b}_{X_{sa}}^t)$ is concentrated in degree $0$. Hence $R\mathcal{H}{om}_{\C_{X_{sa}}}(F,\mathcal{D}{b}_{X_{sa}}^t)\otimes \C_X^{[\hbar]}$ is also concentrated in degree $0$.
\end{proof}

\begin{definition}
The functor of \textit{polynomial extension of tempered distributions}, noted  $TDb^{[\hbar]}(\cdot)$, is the functor   $\text{D}^b_{\R-c}(\C_X^{\hbar})\to D^b(\C_X^{\hbar})$ defined by $$TDb^{[\hbar]}(F):=\rho^{-1}R\mathcal{H}{om}_{\C_{X_{sa}}^{\hbar}}(\rho_*F,\mathcal{D}{b}_{X_{sa}}^{t,[\hbar]}).$$
\end{definition}

 Since  $\mathcal{D}{b}_{X_{sa}}^{t,[\hbar]}$ belongs to $\text{Mod}(\rho_!(\shd_X)^{\hbar})$ it follows that ${TDb}^{[\hbar]}(F)$ is an object of $D^b(\shd_X^{\hbar})$. In other words, $TDb^{[\hbar]}(\cdot)$ is a functor from $\text{D}^b_{\R-c}(\C_X^{\hbar})$ to $D^b(\shd_X^{\hbar}).$

  By Lemma \ref{L:206},   ${TDb}^{[\hbar]}(F)$ is cohomologically $\hbar$-torsion.

\begin{remark}\label{R:20}
Let $F\in D^b_{\R-c}(\C_X)$. Since  $$\rho^{-1}R\mathcal{H}{om}(F, \mathcal{D}b_{X_{sa}}^t)\simeq t\shh om(F,\mathcal{D}b_X),$$  Lemma \ref{L:13} says nothing more  than ${TDb}^{[\hbar]}(F^{\hbar})$ is isomorphic to  $t\shh om(F, \mathcal{D}b_X)\otimes\C_X^{[\hbar]}$. Moreover the isomorphism is compatible with the  structure of $\shd_X^{\hbar}$-modules.

In particular, since $t\shh om(\C_U, \mathcal{D}b_X)\otimes\C_X^{[\hbar]}$, being a $\mathcal{C}_X^{\infty}$-module,  is a soft sheaf, we have that
$R\Gamma(U; \mathcal{D}{b}_{X_{sa}}^{t,[\hbar]})\simeq  R\Gamma(X; t\shh om(\C_U, \mathcal{D}b_X)\otimes \C_X^{[\hbar]})$  is concentrated in degree zero. Namely, if $U$ is relatively compact,
$$\Gamma(U; \mathcal{D}{b}_{X_{sa}}^{t,[\hbar]})\simeq \Gamma(U; \mathcal{D}{b}_{X_{sa}}^t)\otimes \C_X^{[\hbar]}$$ by Proposition \ref{P:1}.

\end{remark}
\begin{remark}\label{L:14}
For any $F\in \text{D}^b_{\R-c}(\C_X^{\hbar})$,
choosing a resolution of $F$ as in (\ref{E:21}), we conclude that ${TDb}^{[\hbar]}(F)$ is isomorphic to a bounded complex $TDb^{[\hbar]}(F^\bullet)$ such that   each entry is isomorphic to a locally finite sum of $\mathcal{C}_X^\infty$-modules of the form $TDb^{[\hbar]}(\C^{\hbar}_U)$, with open subanalytic sets $U$. This provides a soft resolution of $TDb^{[\hbar]}(F)$.
Namely, for any  open subanalytic set $U$ in $X$, we have
 $$R\Gamma_c(U; TDb^{[\hbar]}(F))\simeq \Gamma_c(U; TDb^{[\hbar]}(F^{\bullet}))$$ and  $$R\Gamma (U; TDb^{[\hbar]}(F))\simeq \Gamma (U; TDb^{[\hbar]}(F^{\bullet})).$$
  \end{remark}

We obtain a functor of \textit{polynomial extension of tempered holomorphic functions} $${TH}^{[\hbar]}(\cdot):\text{D}^b_{\R-c}(\C_X^{\hbar})\to\text{D}^b(\shd_X^{\hbar}),$$ by setting $$ \ F\mapsto \rho^{-1}R\mathcal{H}{om}_{\C_{X_{sa}}^{\hbar}}(\rho_*F,\sho_{X_{sa}}^{t, [\hbar]}). $$

Clearly, we have, for $F\in D^b_{\R-c}(\C_X^{\hbar})$,

$${TH}^{[\hbar]}(F)\simeq R\mathcal{H}om_{\shd_{\overline{X}}}(\sho_{\overline{X}}, TDb^{[\hbar]}(F)).$$ In particular, ${TH}^{[\hbar]}(F)$ is cohomologically $\hbar$-torsion.

Therefore, ${TDb}^{[\hbar]}(\cdot)$ can be understood as the polynomial  extension of the functor $t\shh om(\cdot; \mathcal{D}b_X)$ and ${TH}^{[\hbar]}(\cdot)$ can be understood as the polynomial extension of the functor $t\shh om(\cdot; \sho_X)$.

\begin{remark}\label{R:16}
Let $F\in D^b_{\R-c}(\C_X)$. Since one has  $$\rho^{-1}R\mathcal{H}{om}(F,\sho_{X_{sa}}^t)\simeq t\shh om(F,\sho_X),$$  it follows  that
${TH}^{[\hbar]}(F^{\hbar})$ is isomorphic to $t\shh om(F, \sho_X)\otimes \C_X^{[\hbar]}$ hence it is cohomologically $\hbar$-torsion.
In particular, if $U$ is Stein subanalytic relatively compact, and $X$ is a Stein manifold $$R\Gamma(U; \sho_{X_{sa}}^{t, [\hbar]})\simeq R\Gamma(X; TH^{[\hbar]}(\C_U^{\hbar}))$$ is concentrated in degree 0.

All these isomorphisms are compatible with the  structure of  $\shd_X^{\hbar}$-modules.
\end{remark}
As an application we obtain the following comparison result:

\begin{proposition}\label{P:17}
Let $\shm \in D^b_{rh}(\shd_X)$ and let $F\in D^b_{\R-c}(\C^{\hbar}_X)$. Then, the natural morphism:
\begin{equation}\label{E:17}
R\shh{om}_{\shd_X^{\hbar}}(\shm, TH^{[\hbar]}(F))\to
R\shh{om}_{\shd_X^{\hbar}}(\shm, R\shh{om}_{\C_X^{\hbar}}(F, \sho_X^{[\hbar]}))
\end{equation}
is an isomorphism in $D^b(\C_X^{\hbar})$.
\end{proposition}
\begin{proof}

Recall that $TH^{[\hbar]}(F)$ is cohomologically $\hbar$-torsion. We have

$$R\shh{om}_{\shd_X^{\hbar}}(\shm, TH^{[\hbar]}(F))\simeq R\shh{om}_{\shd_X^{\hbar}}(\shm, \shd_X^{\hbar})\otimes_{\shd_X^{\hbar}} TH^{[\hbar]}(F),$$

$$R\shh{om}_{\shd_X^{\hbar}}(\shm, R\shh{om}_{\C_X^{\hbar}}(F, \sho_X^{[\hbar]}))\simeq R\shh{om}_{\shd_X^{\hbar}}(\shm, \shd_X^{\hbar})\otimes_{\shd_X^{\hbar}} R\shh{om}_{\C_X}(F, \sho_X^{[\hbar]}).$$

By lemmas \ref{L:204} and \ref{L:206}, both sides of (\ref{E:17}) are cohomologically $\hbar$-torsion.

The proof then follows by Lemma \ref {L:200} and  the isomorphism $$R\shh{om}_{\shd_X}(\shm, t\shh{om}(F, \sho_X))\to
R\shh{om}_{\shd_X}(\shm, R\shh{om}(F, \sho_X))
 $$ proved in \cite{K2}.
\end{proof}

To end this section, we remark that our perspective here  is to work in the framework of $\C^{\hbar}$-algebras, or sheaves of modules over such algebras. However, the polynomial extensions of $\sho_X$ or of $\shd_X$ as well the associated categories of modules, have their own interest, but we shall not develop here such theory.

Let us just remark some obvious facts.
 $\shd_X^{[\hbar]}$ is a flat $\shd_X$-module and $\sho_X^{[\hbar]}$ is a flat $\sho_X$-module.
We may endow $\shd_X^{[\hbar]}$ with the filtration $F_m(\shd_X^{[\hbar]})$ image of $F_m(\shd_X)[\hbar^{-1}]$ in $\shd_X^{[\hbar]}$, where $ (F_m(\shd_X))_{m\geq 0}$ denotes the filtration on $\shd_X$ by the usual order.






\section{Duality for formal and polynomial extension functors}

\hspace*{\parindent}We shall now state and prove our main  results.

Let $X$ be a real analytic manifold.
As proved in \cite {K2}, Lemma \ref{L:18}, given $F\in D_{\R-c}^b(\C_X)$, for any open subanalytic set $U$, $\Gamma(U; F\overset{w}{\otimes}\mathcal{C}^{\infty}_X)$ is a complex of FN spaces, $\Gamma_c(U; t\mathcal{H}om(F, \mathcal{D}b_X^{\nu}))$ is a complex of DFN spaces and they are dual to each other. From this we get that
$\Gamma(U; \mathcal{C}_X^{\infty, \hbar})$ is a $FN$ topological vector space, $\Gamma_c(U; \mathcal{D}b_X ^{[\hbar]})\simeq \Gamma_c(X; \mathcal{D}b_X)\otimes \C_X^{[\hbar]}$ is a DFN topological vector space, and they are dual to each other.

Let us now consider  $F\in D^b_{\R-c}(\C_X^{\hbar})$. We shall denote  ${TDb}^{[\hbar]}(F)\otimes_{\sha_X} {\sha_X^{\nu}}$ by ${TDb}^{[\hbar]}(F)^{\nu}$  and denote
${TH}^{[\hbar]}(F)\otimes_{\sho_X}\Omega_X$ by ${TH}^{[\hbar]}(F)^{\nu}$.

\begin{proposition}\label{P:21}
Let $F\in D^b_{\R-c}(\C_X^{\hbar})$. Then:

$1$.  $R\Gamma(X; F\overset{\mathrm{w},\hbar}{\otimes}\mathcal{C}_X^\infty)$ is isomorphic to a complex of FN spaces.

$2$. $R\Gamma_c(X; {TDb}^{[\hbar]}(F)^{\nu})$ is isomorphic to a complex of DFN spaces.

$3$.  The complexes respectively described in $(1)$  and  in $(2)$ can be chosen as to be dual to each other.
\end{proposition}
\begin{proof}

$1$. We shall adapt the argument of Proposition 2.2 of \cite{KS2}.

a) By Lemma \ref{L:18} and Proposition 2.2 of (loc.cit), the assertion is true for  $F=\C_U^{\hbar}$, with arbitrary $U$ opens subanalytic, since
$\Gamma(X; (\C_U\overset{\mathrm{w}}{\otimes}\mathcal{C}_X^\infty)^{\hbar})\simeq \Gamma(X; (\C_U\overset{\mathrm{w}}{\otimes}\mathcal{C}_X^\infty))^{\hbar}.$

b) Now assume that $F$ has compact support. Then, by Lemma \ref{L:19}, $F$ is quasi-isomorphic to a bounded complex :

$$F^{\bullet}: \cdots\to F^{-1}\to F^{0}\to 0$$
where $F^0$ is in degree $0$ and each $F^i$ is a finite direct sum of sheaves of type $\C_U^{\hbar}$, $U$ open subanalytic relatively compact. Hence, by a), applying the functor $\Gamma(X; (\cdot)\overset{\text{w},\hbar}{\otimes}\mathcal{C}_X^\infty)$ to the complex above, we obtain an isomorphism $R\Gamma(X; F\overset{\mathrm{w},\hbar}{\otimes}\mathcal{C}_X^\infty)\simeq V^{\bullet}$, where  of $V^{\bullet}$ is a complex of type FN.

c) To treat the general case, let us take an increasing sequence $Z_n$ of compact subanalytic subsets such that $X$ is the union of the interiors of $Z_n$. Let us consider a resolution of $F$, $F^{\bullet}$, given by Lemma \ref{L:19}. For $F_{Z_n}$, note by $F_{Z_n}^{\bullet}$ the associated resolution.

Each entry $F^i$ of $F^\bullet$ being a locally finite direct sum of the form $\oplus_{j_i}\C_{U_{j_i}}^\hbar$, for relatively compact subanalytic open sets $U_{j_i}$, we have: $$F^i\overset{\mathrm{w},\hbar}{\otimes}\mathcal{C}_X^\infty\simeq \oplus_{j_i} (\C_{U_{j_i}}\overset{\mathrm{w}}{\otimes}\mathcal{C}_X^\infty)^\hbar\simeq \prod_{j_i} (\C_{U_{j_i}}\overset{\mathrm{w}}{\otimes}\mathcal{C}_X^\infty)^\hbar,$$ since the direct sum of a locally finite family of sheaves on $X_{sa}$ is isomorphic to the product of the same family. Hence, we get: $$\Gamma(X;F^i\overset{\mathrm{w},\hbar}{\otimes}\mathcal{C}_X^\infty)\simeq \prod_{j_i}\Gamma(X;\C_{U_{j_i}}\overset{\mathrm{w}}{\otimes}\mathcal{C}_X^\infty)^\hbar\simeq$$ $$ \prod_{j_i}\left(\underset{n}{\underset{\longleftarrow}{\lim}}\Gamma(X;\C_{U_{j_i}\cap Z_n}\overset{\mathrm{w}}{\otimes}\mathcal{C}_X^\infty)^\hbar\right)\simeq
\underset{n}{\underset{\longleftarrow}{\lim}}\Gamma(X;F^i_{Z_n}\overset{\mathrm{w},\hbar}{\otimes}\mathcal{C}_X^\infty).$$

As a consequence, the complex $R\Gamma(X; F\overset{\mathrm{w},\hbar}{\otimes}\mathcal{C}_X^\infty)$ is isomorphic to  the projective limit  of the complexes $\Gamma(X; F^{\bullet}_{Z_n}\overset{\mathrm{w},\hbar}{\otimes}\mathcal{C}_X^\infty)$ hence it is of FN type.

 $2$. We shall use a similar argument as in $1$.

 a') Let $U$ open subanalytic in X. The assertion for $F=\C^{\hbar}_U$  is an immediate consequence of Lemma \ref{L:13}.

 b') Assume that $F$ is compactly supported and choose a resolution $F^{\bullet}$ of $F$  as in b) and apply Remark \ref{L:14}. Hence  $R\Gamma_c(X; {TDb}^{[\hbar]}(F)^{\nu})$ is isomorphic to  $\Gamma_c(X; {TDb}^{[\hbar]}(F^{\bullet})^{\nu})$. Since $\Gamma_c(X; \cdot)$ commutes with inductive limits, this entails that $R\Gamma_c(X; {TDb}^{[\hbar]}(F)^{\nu})\simeq W^{\bullet}$ with $W^{\bullet}$ a complex of type DFN.

 c') With the same choice of $Z_n$ as in c),  note that $TDb^{[\hbar]}(F)^{\nu}$ is represented by the inductive limit of $TDb^{[\hbar]}(F^{\bullet}_{Z_n})^{\nu}$. In fact, each entry $F^i$ of $F^\bullet$ is a locally finite direct sum of the form $\oplus_{j_i}\C_{U_{j_i}}^\hbar$, for subanalytic open sets $U_{j_i}$, and we have: $$\Gamma_c(X; {TDb}^{[\hbar]}(F^i)^{\nu})\simeq \oplus_{j_i}\Gamma_c(X;{TDb}^{[\hbar]}(\C_{U_{j_i}})^{\nu})\simeq$$ $$\simeq\oplus_{j_i}\left(\Gamma_c(X;t\shh{om}(\C_{U_{j_i}}, \mathcal{D}b)^{\nu})\otimes\C^{[\hbar]}_X\right)\simeq $$ $$\simeq \oplus_{j_i}\left(\underset{n}{\underset{\longrightarrow}{\lim}}\Gamma_c(X;t\shh{om}(\C_{U_{j_i}\cap Z_n}, \mathcal{D}b_X)^{\nu})\otimes\C^{[\hbar]}_X\right)\simeq\underset{n}{\underset{\longrightarrow}{\lim}}\Gamma_c(X; {TDb}^{[\hbar]}(F^i_{Z_n})^{\nu}),$$
which entails that $\Gamma_c(X; {TDb}^{[\hbar]}(F)^{\nu})$ is isomorphic to a complex of type DFN.

$3$. Let us choose an almost free resolution of $F$, $F^{\bullet}$, as above, and let us choose an increasing sequence $Z_n$ of compact subanalytic subsets such that $X$ is the union of the interiors of $Z_n$. By Lemma 4.3 of \cite{KS2}, for each $F^i$, $\Gamma(X;F^i_{Z_n}\overset{\mathrm{w},\hbar}{\otimes}\mathcal{C}_X^\infty)$ is FN, $\Gamma_c(X; {TDb}^{[\hbar]}(F^i_{Z_n})^{\nu})$ is DFN and they  are dual to each other. The proof follows by remarking that the topological dual of
$\underset{n}{\underset{\longleftarrow}{\lim}}\Gamma(X;F^i_{Z_n}\overset{\mathrm{w},\hbar}{\otimes}\mathcal{C}_X^\infty)$ is $\underset{n}{\underset{\longrightarrow}{\lim}}\Gamma_c(X; {TDb}^{[\hbar]}(F^i_{Z_n})^{\nu})$ and conversely.

\end{proof}

Let  now $X$ be a complex manifold of dimension $d_ X$, let $\mathcal{C}_X^{\infty, (0, \bullet)}$ denote the complex of conjugate differential forms with coefficients in $\mathcal{C}_{X}^{\infty}$ and  recall that $$ F\overset{\text{w},\hbar}{\otimes}\mathcal{C}_X^{\infty,(0, \bullet)}:=(F\overset{\text{w},\hbar}{\otimes}\mathcal{C}^{\infty}_X)\otimes_{\sha_X}\mathcal{C}_{X}^{\infty,(0, \bullet)}.$$

Then, for any $F\in D^b_{\R-c}(\C_X^\hbar)$,
\begin{equation}\label{E:20}
F\overset{\text{w},\hbar}{\otimes}\sho_X\simeq F\overset{\text{w},\hbar}{\otimes}\mathcal{C}_X^{\infty,(0, \bullet)}.
\end{equation}
Indeed, by taking an almost free resolution of $F$, we may assume that
$F=G^{\hbar}$, with $G\in \text{Mod}_{\R-c}(\C_X)$. In that case,

$$F\overset{\text{w},\hbar}{\otimes}\sho_X\simeq G^\hbar\overset{\text{w},\hbar}{\otimes}\sho_X\simeq (G\overset{\text{w}}{\otimes}\sho_X)^{\hbar}\simeq (G\overset{\text{w}}{\otimes}\mathcal{C}_X^{\infty,(0, \bullet)})^{\hbar}\simeq G^{\hbar}\overset{\text{w}}{\otimes}\mathcal{C}_X^{\infty,(0, \bullet)}.$$
Similarly,  consider the resolution $\mathcal{D}b_{X}^{(n, n-\bullet)}$ of $\Omega_{X}[d_X]$.
We set
$$TDb^{[\hbar]}(F)\otimes_{\sha_X}\mathcal{D}b_{X}^{(n, n-\bullet)}:=TDb^{[\hbar]}(F)^{(n, n-\bullet)}.$$
Therefore, we get
\begin{equation}\label{E:22}
TH^{[\hbar]}(F)^{\nu}[d_X]\simeq TDb^{[\hbar]}(F)^{(n, n-\bullet)}.
\end{equation}

\begin{theorem}\label{T:22}
Let $\shm \in D^b_{coh}(\shd_X^{\hbar})$ and let $F, G\in D^b_{\R-c}(\C_X^{\hbar})$.
Then:
\begin{enumerate}
\item{We can define
$R\Gamma(X; R\mathcal{H}{om}_{\shd_X^{\hbar}}(\shm\overset{L}{\otimes_{\C^{\hbar}}} G, F\overset{\mathrm{w},\hbar}{\otimes}\sho_X))$ as an object of $D^b(FN)$, functorially with respect to $\shm$, $F$ and $G$.}

\item {We can define
 $R\Gamma_c(X;{TH}^{[\hbar]}(F)^{\nu}[d _X]\overset{L}{\otimes_{{\shd_X}^{\hbar}}}(\shm\overset{L}{\otimes_{\C^{\hbar}}} G))$ as an object of $D^b(DFN)$ functorially with respect to $\shm$, $F$ and $G$.}

 \item{The objects described in (1) and (2) can be constructed in such a way that they are dual to each other.}
\end{enumerate}
\end{theorem}
\begin{proof}
We shall essentially follow the method of the proof of Theorem 6.1 of \cite{KS2}.

 Applying  Proposition \ref{P:201} let us choose $\mathcal{L}(\shm)$ an almost free resolution of $\shm$, and  by Lemma \ref{L:19}, let us choose $\mathcal{L}(G)$, $\mathcal{L}(F)$ almost free resolutions of $G$ and $F$ respectively.

1. Note that, for any $H\in \text{Mod}_{\R-c}(\C_X)$  and any $k\geq 0$,
$$\Gamma(X; \shh om_{\shd^{\hbar}_X}(\mathcal{L}(\shm)\otimes_{\C^{\hbar}_X} \mathcal{L}(G), H^{\hbar}\overset{\text{w},\hbar}{\otimes} \mathcal{C}_X ^{\infty, (0, k)}))\simeq \Pi_{i,j}\Gamma(U_i\cap V_j; H^{\hbar}\overset{\text{w},\hbar}{\otimes} \mathcal{C}_X ^{\infty, (0, k)}),$$ for adequate choice of open subanalytic relatively compact open sets $U_i$ and $V_j$ (defined by $\mathcal{L}(\shm)$ and $\mathcal{L}(G)$).

We have
\begin{equation}\label{E:216}
R\Gamma(X; R\mathcal{H}{om}_{\shd_X^{\hbar}}(\shm\overset{L}{\otimes_{\C^{\hbar}_X}} G, F\overset{\text{w},\hbar}{\otimes}\sho_X))\simeq \Gamma(X; \shh om_{\shd_X^{\hbar}}(\mathcal{L}(\shm)\otimes_{\C^{\hbar}_X}\mathcal{L}(G),
\mathcal{L}(F)\overset{\text{w},\hbar}{\otimes} \mathcal{C}_X ^{\infty, (0, \bullet)}))
\end{equation}

The result then follows by $1.$ of Proposition \ref{P:21}.

2. The proof goes similar to the proof of 1. Replace $TH^{[\hbar]}(F)^{\nu}[d_X]$ by the isomorphic complex $TDb^{[\hbar]}(F)^{(n, n-\bullet)}$.

 Noting that, for any $H\in \text{Mod}_{\R-c}(\C_X)$, $$\Gamma_c(X; TDb^{[\hbar]}(H^{\hbar})^{(n, n-k)}\otimes_{\shd^{\hbar}_X}(\mathcal{L}(\shm)\otimes_{\C^{\hbar}_X} \mathcal{L}(G)))\simeq\oplus_{i, j}\Gamma_c(U_i\cap V_j; TDb^{[\hbar]}(H^{\hbar})^{(n, n-k)})$$ the proof follows by 2. of Proposition \ref{P:21}.

 3. The proof follows straightforwardly from 1. and 2. above and 3. of Proposition \ref{P:21}.
\end{proof}

\begin{example} Let us consider the Example 8.5 of \cite {DGS}.
Let $X$ be $\R$ with the coordinate $x$ and consider the coherent $\shd_X^{\hbar}$-module defined by the equation $x-\hbar\partial_x$, $\shm=\shd_X^{\hbar}/\langle x-\hbar \partial _x\rangle.$ As proved in loc.cit, $\shm$ is isomorphic as a $\shd_X^{\hbar}$-module to $\shn=\shd_X^{\hbar}/\langle x\rangle.$

Let $F=\C_Z$, where $Z=\{0\}$. Then $\C_{\{0\}}^{\hbar}\overset{\text{w},\hbar}{\otimes}\mathcal{C}_X^{\infty}$ is given by the exact sequence $$   0\to ({\mathcal{I}^{\infty}_{X, \{0\}}})^{\hbar}\to \mathcal{C}_X^{\infty, \hbar}\to \C_{\{0\}}^{\hbar}\overset{\text{w},\hbar}{\otimes}\mathcal{C}_X^{\infty}\to 0.$$ Recall that a section of $\C_{\{0\}}^{\hbar}\overset{\text{w} }{\otimes}\mathcal{C}_X^{\infty}$ is given by $(\lambda_j)_{j\geq 0}, \lambda_j\in\C$ corresponding to $f\in \mathcal{C}_X^{\infty}$ satisfying $\partial_x^{(j)} (f)(0)=\lambda_j$.
The action of $x$ on $\C_{\{0\}}^{\hbar}\overset{\text{w},\hbar}{\otimes}\mathcal{C}_X^{\infty}\simeq (\C_{\{0\}}^{\hbar}\overset{\text{w}}{\otimes}\mathcal{C}_X^{\infty})^{\hbar}$ is induced by the action of $x$ in $\C_{\{0\}}\overset{\text{w}}{\otimes}\mathcal{C}_X^{\infty}$ which is given by $x(\lambda_j)=(\mu_j)=(j\lambda_{j-1})$ for $j\geq 0$. Therefore, $$\Gamma(X; \shh\text{om}_{\shd_X^{\hbar}}(\shm, \C_{\{0\}}^{\hbar}\overset{\text{w},\hbar}{\otimes}\mathcal{C}_X^{\infty}))=0$$ and $$\Gamma(X; \she xt^1_{\shd_X^{\hbar}}(\shm, \C_{\{0\}}^{\hbar}\overset{\text{w},\hbar}{\otimes}\mathcal{C}_X^{\infty}))\simeq \C^{\hbar}.$$
On the other hand, the action of $x$ on $\Gamma_{\{0\}}(\mathcal{D}b)^{[\hbar]}$ is induced by the action of $x$ on $\Gamma_{\{0\}}(\mathcal{D}b_X)$ which is given by $x\sum_{i\geq 0} \lambda_i\delta^{(i)}(x)=\sum_{i\geq 1}\lambda_i(-i)\delta^{(i-1)}(x)$.

Hence
$$\Gamma(X; \shh\text{om}_{\shd_X^{\hbar}}(\shm, \Gamma_{\{0\}}(\mathcal{D}b_X^{[\hbar]})))\simeq (\C\delta(x))^{[\hbar]}$$ and $\Gamma(X; \she xt^1_{\shd_X^{\hbar}}(\shm, \Gamma_{\{0\}}(\mathcal{D}b_X^{[\hbar]})))=0$.
\end{example}

{\small

Ana Rita Martins\\ Faculdade de Engenharia da Universidade Cat\'{o}lica Portuguesa,\\ Estrada Oct\'{a}vio Pato, 2635-631 Rio-de-Mouro
 Portugal\\ ritamartins@fe.ucp.pt\newline

Teresa Monteiro Fernandes\\ Centro de {\'A}lgebra da Universidade
de Lisboa e Departamento de Matem\' atica da FCUL, Complexo 2,\\ 2 Avenida Prof. Gama Pinto, 1699 Lisboa
 Portugal\\ tmf@ptmat.fc.ul.pt}

\end{document}